\documentclass[leqno,12pt]{article} 
\usepackage{amsmath, amssymb}
\usepackage{amsthm} 
\usepackage{amssymb}
\usepackage{mathrsfs}
\usepackage{amssymb, url, color, graphicx, amscd, mathrsfs}
\usepackage[colorlinks=true, bookmarks=true, pdfstartview=FitH, pagebackref=true, linktocpage=true, linkcolor = magenta, citecolor = blue]{hyperref}
\usepackage{graphicx}
\usepackage{times}
\theoremstyle{plain} 
\newtheorem{theorem}{\indent\sc Theorem}[section]
\newtheorem{lemma}[theorem]{\indent\sc Lemma}

\newtheorem{proposition}[theorem]{\indent\sc Proposition}

\theoremstyle{definition} 

\numberwithin{equation}{section}
\makeatletter

\def\address#1#2{\begingroup
\noindent\parbox[t]{7.8cm}{
\small{\scshape\ignorespaces#1}\par\vskip1ex
\noindent\small{\itshape E-mail address}
\/: #2\par\vskip4ex}\hfill
\endgroup}

\makeatother
\title{Rigidity and vanishing results on totally real submanifolds under $L^p$-integrable conditions} 
\author{\textsc{N. T. Dung, L. G. Linh, P. B. Ngan and A. Upadhyay} 
}
\date{} 
\begin{document}

\maketitle


\begin{abstract}
\noindent
In this paper, we revise some results on rigidity and vanishing properties obtained by \textit{Cuong et.al} in \cite{CDS24} on $n$-dimensional totally real minimal submanifolds $M$ immersed in complex space forms $\widetilde{M}^n(c)$, for $c\leq0$. We extend the range of $p$ in their paper.
\end{abstract}

\section{Introduction}
In this paper, we consider $M$ to be an $n$-dimension Riemannian manifold immersed in a complex space form. Here  a \textit{complex space form} is a K\"{a}hler manifold of constant holomorphic sectional curvature. Throughout this paper, we will consider $\widetilde{M}_{n+p}({c}), p\geq1$ to be an $(n+p)$-dimensional complex space form of constant holomorphic sectional curvature ${c}$, for $c\leq0$. Assume that $\widetilde{J}$ is the almost complex structure on $\widetilde{M}$. It is well-known that if $M$ admits an isometric immersion into $\widetilde{M}$ such that for all $x$, $\widetilde{J}(T_x(M))\subset \nu_x$, where $T_x(M)$ denotes the tangent space of $M$ at $x$ and $\nu_x$ the normal space at $x$ then $M$ is said to be an \textit{totally real} submanifold of $\widetilde{M}$. We note that if $c=-1$ then sectional curvatures of $\widetilde{M}$ are in $\left[-1,-\frac{1}{4}\right]$ and if $c=0$ then $\widetilde{M}$ is flat (see \cite{AW14, CO74, FFR93}). Hence, in our setting, sectional curvatures of $\widetilde{M}$ are always non-positive. As a consequence, $M$ has a Sobolev inequality.  
\begin{theorem}[\cite{HS74, MS73}]\label{sobolev}
Let $M^n$ be a complete immersed 
submanifold in a nonpositively curved manifold $N^{n+p}, n \geq 3$. Then, for any
$\varphi \in W^{1,2}_0(M)$, we have
$$\left(\int_M|\varphi|^{2n/(n-2)}\right)^\frac{n-2}{n}\leq C_s\int_M|\nabla\varphi|^2+H^2\varphi^2,$$
where $C_s$ is the Sobolev constant which depends only on $n$ and $H$ stands for the mean curvature of $M$.
\end{theorem}
We refer the interested readers to \cite{CHL, CO74, Nor, Gray3,Houh75, Kon08, Nai3, Nai4,Ohnita,Sun93,YK76a,YK76b} for further discussion on totally real submanifolds and related topics. 
On a complete non-compact Riemannian manifold $M$, let $\lambda_1(M)$ be the first eigenvalue (or the bottom of spectrum) which is defined by $\lambda_1(M) = \inf\limits_{\Omega}\lambda_1(\Omega)$ where the infimum is taken over all compact domains $\Omega$ in $M$ and  $\lambda_1(\Omega) > 0$ is the first eigenvalue of the Dirichlet boundary value problem 
$$\begin{cases}
	\Delta f+\lambda f = 0 &\text{ in } \Omega,\\
	f = 0 &\text{ on } \partial\Omega.
\end{cases}
$$

Using the Sobolev inequality \eqref{sobolev}, Simon type inequalities, and the Bochner technique, in \cite{CDS24}, the first author and his collaborators obtained the following rigidity result.
\begin{theorem}\label{rigid}
	Suppose that $M$ is an $n$-dimensional totally real minimal submanifold immersed in $\widetilde{M}_n(c)$, where $c\in\{-1, 0\}$ and $n\geq6$. Let $A$ be the second fundamental form of $M$. For any $p\geq1$, if one of the following assumptions holds true 
	\begin{enumerate}
		\item $c=-1$, and $$
		\lambda_1(M)
		>\dfrac{p(n+1)}{4\Big[2-\frac{n-2}{np}\Big]}
		, \quad 
		\|A\|_n^2<\frac{2-\frac{n-2}{np}-\frac{p(n+1)}{4\lambda_1(M)}}{pC_s\left(2-\frac{1}{n}\right)};
		$$
		\item $c=0$ and 
		$$
		\|A\|_n^2<\frac{2-\frac{n-2}{np}}{pC_s\left(2-\frac{1}{n}\right)},
		$$
	\end{enumerate}
	then $M$ is totally geodesic provided that $\|A\|_{2p}<\infty$, where $$\|A\|_{k}=\left(\int_M|A|^{k}\right)^{1/k}, k\in\{n, 2p\}.$$
\end{theorem}
Moreover, they also proved a vanishing result as follows.
\begin{theorem}\label{ends}
	Suppose that $M$ is an $n$-dimensional totally real minimal submanifold immersed in $\widetilde{M}_n(c)$, where $c\in\{-1, 0\}$. Let $M$ be the second fundamental form of $M$. If one of the following assumptions holds true 
	\begin{enumerate}
		\item $c=-1, 1\leq p<\frac{n-2}{2}, n\geq5$, and 
		$$\|A\|_n<\sqrt{\frac{(2p-1)(n-2p-2)}{p^2(n-1)C_s}};$$
		\item $c=0, p\geq1, n\geq2$, and 
		$$\|A\|_n<\sqrt{\frac{2p(n-1)-(n-2)}{(n-1)p^2C_s}},$$
	\end{enumerate}
	then there is no nontrivial $L^{2p}$ harmonic $1$-forms on $M$. Consequently, if either $c=-1$ and $\|A\|_n<\sqrt{\frac{n-4}{(n-1)C_s}}$ or; $c=0$ and $\|A\|_n<\sqrt{\frac{n}{(n-1)C_s}}$, then $M$ has at most one parabolic end.
\end{theorem}

\noindent
Observe that in the above theorems, the authors assumed that $\|A\|_{2p}<\infty$ for rigidity results and $|\omega|_{2p}<\infty$ for vanishing results of harmonic forms $\omega$, and $p\geq1$. Nothe that since $p\geq1$ we must have $2p\geq2$. Therefore, it is natural for us to ask if is it possible to obtain the same results when $0<2p<2$.
Inspired by \cite{CDS24}, our aim in this paper is to extend their results for a wider range of $p$ in $L^{2p}$-integrable conditions. Using a trick by Yong Luo in \cite{YL}, we show that the range of $p$ is extendable. 
Throughout this paper, $C_s$ stands for the Sobolev constant in Theorem \ref{sobolev}. We now can state our first result as follows.
\begin{theorem}\label{rigid}
Let $M$ be an $n$-dimensional totally real minimal submanifold immersed in $\widetilde{M}_n(c)$, where $c\in\{-1, 0\}$. If one of the following assumptions holds true 
\begin{enumerate}
    \item $c=-1, 1-\frac{2}{n}<p<2$, and $$
\lambda_1(M)
>\dfrac{(n+1)(4-p)^2}{16\Big(p+\frac{2}{n}-1\Big)}
, \quad 
\|A\|_n^2<\frac{8\Big(p+\frac{2}{n}-1\Big)-\frac{(n+1)(4-p)^2}{2\lambda_1(M)}}{C_s\left(4-\frac{2}{n}\right)(4-p)^2};
$$
    \item $c=0, 1-\frac{2}{n}<p<2$, and 
  $$
\|A\|_n^2<\frac{8\Big(p+\frac{2}{n}-1\Big)}{C_s\left(4-\frac{2}{n}\right)(4-p)^2},
$$
\end{enumerate}
then $M$ is totally geodesic provided that $\|A\|_{p}<\infty$, where $$\|A\|_{k}=\left(\int_M|A|^{k}\right)^{1/k}, k\in\{n, p\}.$$
\end{theorem}
\noindent
We note that the lower bound of $\lambda_1(M)$ is
$\lambda_1(M)=\frac{(n-1)^2}{16}$ (for instance, see \cite{BM03}).
Hence, if  $1<p<2$ then for $n$ large enough we have
$$8\Big(p+\frac{2}{n}-1\Big)-\frac{8(n+1)(4-p)^2}{(n-1)^2}>0.$$
In fact, we can point out that when $n\geq 6$, we always can find $p\in\left(1-\frac{2}{n};2\right)$ such that 
$$8\Big(p+\frac{2}{n}-1\Big)-\frac{8(n+1)(4-p)^2}{(n-1)^2}>0.$$
\noindent The second result is a vanishing theorem for harmonic forms on $M$.
\begin{theorem}\label{ends}
Let $M$ be an $n$-dimensional totally real minimal submanifold immersed in $\widetilde{M}_n(c)$, where $c\in\{-1, 0\}.$ If one of the following assumptions holds true 
\begin{enumerate}
    \item $c=-1,~ 1<p<2,~ \frac{n-(4-p)^2+(p-2)(n-1)}{n-1} >0$, and 
    $$\|A\|_n<\sqrt{\frac{4n-4(4-p)^2+4(p-2)(n-1)}{(n-1)(4-p)^2C_s}};$$
    \item $c=0, n\geq2, \frac{n-2}{n-1}<p<2$, and 
    $$\|A\|_n<\sqrt{\frac{4n+4(p-2)(n-1)}{(n-1)(4-p)^2C_s}},$$
\end{enumerate}
then there is no nontrivial $L^{p}$ harmonic $1$-forms on $M$. 
\end{theorem}
It is easy to see that
$$\lim\limits_{n\to\infty}\frac{n-(4-p)^2+(p-2)(n-1)}{n-1}=p-1>0.$$
Hence our condition in the first part of Theorem \ref{ends} is available if $n$ is large enough. In fact, for any $n\geq 6$, we can find $p\in(1; 2)$ such that
$$\frac{n-(4-p)^2+(p-2)(n-1)}{n-1}>0.$$

The paper is organized as follows. Besides the introduction part, we will recall some basic facts and give proofs of our results in Section 2. We would like to emphasize that our proof are different from those in \cite{CDS24} but we need to use a trick given by Yong Luo in \cite{YL} (see also \cite{DDY}). 


\section{Rigidity theorems on totally real submanifolds}
In this section, we give a proof of Theorem \ref{rigid} and its consequences. Let $A$ be the second fundamental form of $M$. It is well-known that the following Simon type inequality (see \cite{CDS24}) holds true.
\begin{equation}\label{simon}
	\frac{1}{2}\Delta|A|^2\geq\frac{n+2}{n}|\nabla|A||^2-\left(2-\frac{1}{n}\right)|A|^4+\frac{1}{4}(n+1)c|A|^2.
\end{equation}
To derive our result, we use the following estimate for the bottom of spectrum $\lambda_1(M)$.
\begin{theorem}[\cite{BM03}]\label{lambda}
	Let $N$ be an $n$-dimensional
	complete simply connected Riemannian manifold with sectional curvature $K_N$
	satisfying $K_N \leq - a^2$ for a positive constant $a > 0$. Let $M$ be an $m$-dimensional
	complete noncompact submanifold with bounded mean curvature vector $H$ in $N$
	satisfying $|H| \leq b < (m - 1)a$. Then
	$$\lambda_1(M)\geq\frac{[(m-1)a-b]^2}{4}.$$
\end{theorem}
Finally, we also need a basic inequality.
\begin{lemma}\label{abc}
	Given $a, b, c\in\mathbb{R}$. For any $\varepsilon, \eta>0$, we have
	$$(a+b+c)^2\leq\left(1+\frac{1}{\varepsilon}\right)a^2+(1+\varepsilon)(1+\eta)b^2+(1+\varepsilon)\left(1+\frac{1}{\eta}\right)c^2.$$
	\end{lemma}
\begin{proof}
	It is easy to see that 
	$$(a+b+c)^2\leq\left(1+\frac{1}{\varepsilon}\right)a^2+(1+\varepsilon)(b+c)^2,$$
	and
	$$(b+c)^2\leq\left(1+\eta\right)b^2+(1+\frac{1}{\eta})c^2.$$
	The proof follows by combining the above inequalities.
\end{proof}
Now, we give a proof of the first main theorem.
\begin{proof}[Proof of Theorem \ref{rigid}]
Let $\varphi$ be a Lipschitz function with compact support in a geodesic ball $B_o(3R)$ of radius $3R$ centered at a point $o\in M$ and $\delta$ a positive constant. Multiplying both side of \eqref{simon} by $\varphi^2\left(\left|A\right|+\delta\right)^{p-2}$ and integrating over $M$, we obtain
\begin{align*}
\int_{M}\varphi^2\left(\left|A\right|+\delta\right)^{p-2}\Delta\left|A\right|^2-\dfrac{c}{2}(n+1) &\int_{M}\varphi^2\left(\left|A\right|+\delta\right)^{p-2}\left|A\right|^{2}\\+\left(4-\dfrac{2}{n}\right)\int_{M}\varphi^2\left(\left|A\right|+\delta\right)^{p-2}\left|A\right|^{4}
&\geq \left(2+\dfrac{4}{n}\right)\int_{M}\varphi^2\left(\left|A\right|+\delta\right)^{p-2}\left|\nabla\left|A\right|\right|^2.
\end{align*}
The first term of the above inequality can be directly  computed as follows.
\begin{align*}
\int_{M}&\varphi^2\left(\left|A\right|+\delta\right)^{p-2}\Delta\left|A\right|^2\\
&=-\int_{M}\left\langle\nabla\left[\varphi^2\left(\left|A\right|+\delta\right)^{p-2}\right], \nabla \left|A\right|^2\right\rangle\\
&=-\int_{M}\left\langle\nabla\left(\varphi^2\right)\left(\left|A\right|+\delta\right)^{p-2}+\nabla\left[\left(\left|A\right|+\delta\right)^{p-2}\right]\varphi^2, \nabla\left|A\right|^2\right\rangle\\
&=-\int_{M}\left\langle 2 \varphi\nabla \varphi\left|A\right|\left(\left|A\right|+\delta\right)^{p-2}+\nabla\left[\left(\left|A\right|+\delta\right)^{p-2}\right]\varphi^2, 2\left|A\right| \nabla \left|A\right|\right\rangle\\
&=-\int_{M}\left\langle 2 \varphi\nabla \varphi\left|A\right|\left(\left|A\right|+\delta\right)^{p-2}+\left[\left(p-2\right)\left(\left|A\right|+\delta\right)^{p-3}\nabla\left(\left|A\right|+\delta\right)\right]\varphi^2, 2\left|A\right| \nabla \left|A\right|\right\rangle\\
&=-\int_{M}\left\langle 2 \varphi\nabla \varphi\left|A\right|\left(\left|A\right|+\delta\right)^{p-2}+\left[\left(p-2\right)\left(\left|A\right|+\delta\right)^{p-3}\nabla\left|A\right|\right]\varphi^2, 2\left|A\right| \nabla \left|A\right|\right\rangle\\
&=-2\left(p-2\right)\int_{M}\varphi^2\left|A\right|\left(\left|A\right|+\delta\right)^{p-3}\left|\nabla\left|A\right|\right|^2-4\int_{M}\varphi\left|A\right|\left(\left|A\right|+\delta\right)^{p-2}\left\langle \nabla\varphi, \nabla\left|A\right|\right\rangle.
\end{align*}
Plugging this identity into the previous inequality, it turns out that
\begin{equation}\label{eq:7d}
\begin{split}
\left(2+\dfrac{4}{n}\right)&\int_{M}\varphi^2\left(\left|A\right|+\delta\right)^{p-2}\left|\nabla\left|A\right|\right|^2\\
&\leq -\dfrac{c}{2}(n+1)\int_{M}\varphi^2\left(\left|A\right|+\delta\right)^{p-2}\left|A\right|^{2}\\
&\hspace{0.44cm}+\left(4-\dfrac{2}{n}\right)\int_{M}\varphi^2\left(\left|A\right|+\delta\right)^{p-2}\left|A\right|^{4}\\
&\hspace{0.43cm}-2\left(p-2\right)\int_{M}\varphi^2\left|A\right|\left(\left|A\right|+\delta\right)^{p-3}\left|\nabla\left|A\right|\right|^2\\
&\hspace{0.43cm}-4\int_{M}\varphi\left|A\right|\left(\left|A\right|+\delta\right)^{p-2}\left\langle \nabla\varphi, \nabla\left|A\right|\right\rangle.
\end{split}
\end{equation}
For any $\varepsilon, \eta>0$, using H$\ddot{o}$lder's inequality, the weighted Sobolev's inequality and Cauchy-Schwarz's inequality consecutively, we estimate the second term on the right hand side of \eqref{eq:7d} as follows.
\begin{equation}\label{eq:7}
\begin{split}
\int_{M}&\varphi^2\left(\left|A\right|+\delta\right)^{p-2}\left|A\right|^{4}\\
&\leq \Bigg(\int_{M}\left|A\right|^{2\cdot\frac{n}{2}}\Bigg)^{\frac{2}{n}}\Bigg\{\int_{M}\left[\left|A\right| \varphi\left(\left|A\right|+\delta\right)^{\frac{p-2}{2}}\right]^{2\cdot\frac{n}{n-2}}\Bigg\}^{\frac{n-2}{n}}\\
&\leq C_s\|A\|_n^2\int_M\left|\nabla\left[|A|\varphi\left(\left|A\right|+\delta\right)^{\frac{p-2}{2}}\right]\right|^2\\
&=C_s\|A\|_n^2\int_M \Bigg|\nabla\varphi|A|\left(\left|A\right|+\delta\right)^{\frac{p-2}{2}}+\varphi\nabla|A|\left(\left|A\right|+\delta\right)^{\frac{p-2}{2}}\\
&\hspace{0.6cm}\frac{p-2}{2}\varphi|A|\left(\left|A\right|+\delta\right)^{\frac{p-4}{2}}\nabla\left|A\right|\Bigg|^2\\
&\leq C_s\|A\|_n^2\Bigg[\left(1+\frac{1}{\varepsilon}\right)\int_M\left|\nabla\varphi\right|^2\left|A\right|^2\left(\left|A\right|+\delta\right)^{p-2}\\
&\hspace{0.5cm}+\left(1+\varepsilon\right)\left(1+\eta\right)\int_M\varphi^2\left|\nabla\left|A\right|\right|^2\left(\left|A\right|+\delta\right)^{p-2}\\
&\hspace{0.5cm}+\left(1+\varepsilon\right)\left(1+\frac{1}{\eta}\right)\frac{(p-2)^2}{4}\int_M \varphi^2\left|A\right|^2\left(\left|A\right|+\delta\right)^{p-4}\left|\nabla\left|A\right|\right|^2\Bigg]\\
&\leq C_s\|A\|_n^2\left(1+\frac{1}{\varepsilon}\right)\int_M\left|\nabla\varphi\right|^2\left(\left|A\right|+\delta\right)^{p-2}\left|A\right|^2\\
&\hspace{0.5cm}+C_s\|A\|_n^2\left(1+\varepsilon\right)\Bigg[\left(1+\eta\right)+\left(1+\frac{1}{\eta}\right)\frac{(p-2)^2}{4}\Bigg]\\
&\hspace{0.7cm}\int_M\varphi^2\left(\left|A\right|+\delta\right)^{p-2}\left|\nabla\left|A\right|\right|^2.
\end{split}
\end{equation}
Here in the fourth inequality, we used Lemma \ref{abc}. Now, the definition of $\lambda_1(M)$ and its variational characterization (see \cite{BM03}) imply
$$
\lambda_1(M)\leq  \dfrac{\int_{M}\left|\nabla \varphi\right|^2}{\int_{M}\varphi^2},
$$
for any compactly supported nonconstant Lipschitz function $\varphi$ on $M$. This together with Lemma \ref{abc} infers
\begin{equation}\label{eq:10}
\begin{split}
\lambda_1(M)
&\int_{M}\left|A\right|^{2}\varphi^2\left(\left|A\right|+\delta\right)^{\frac{p-2}{2}}\\
&\leq\int_M\left|\nabla\left[|A|\varphi\left(\left|A\right|+\delta\right)^{\frac{p-2}{2}}\right]\right|^2\\
&\leq \left(1+\frac{1}{\varepsilon}\right)\int_M\left|\nabla\varphi\right|^2\left(\left|A\right|+\delta\right)^{p-2}\left|A\right|^2\\
&\hspace{0.5cm}+\left(1+\varepsilon\right)\Bigg[\left(1+\eta\right)+\left(1+\frac{1}{\eta}\right)\frac{(p-2)^2}{4}\Bigg]\int_M\varphi^2\left(\left|A\right|+\delta\right)^{p-2}\left|\nabla\left|A\right|\right|^2.
\end{split}
\end{equation}
Observar that 
\begin{equation}\label{eq:12}
\int_M\varphi^2\left(\left|A\right|+\delta\right)^{p-2}\left|\nabla\left|A\right|\right|^2 \geq\int_M\varphi^2\left(\left|A\right|+\delta\right)^{p-3}\left|A\right|\left|\nabla\left|A\right|\right|^2. 
\end{equation}
and a simple application of the Cauchy inequality gives
\begin{equation}\label{eq:13}
\begin{split}
-2\int_{M}&\varphi\left(\left|A\right|+\delta\right)^{p-2}\left|A\right|\langle\nabla \varphi,\nabla\left|A\right|\rangle\\
&\leq\frac{1}{\epsilon}\int_{M}\left(\left|A\right|+\delta\right)^{p-2}\left|A\right|^{2}\left|\nabla \varphi\right|^2
+\epsilon\int_M\varphi^2\left(\left|A\right|+\delta\right)^{p-2}\left|\nabla\left|A\right|\right|^2  
\end{split}
\end{equation}
for any $\epsilon>0$. Combining \eqref{eq:7d}-\eqref{eq:13}, we obtain
\begin{align*}
&\left(2+\dfrac{4}{n}\right)\int_{M}\varphi^2\left(\left|A\right|+\delta\right)^{p-2}\left|\nabla\left|A\right|\right|^2\\
&\hspace{1cm}\leq\Bigg[ C_s\|A\|_n^2\Bigg(4-\frac{2}{n}\Bigg)\Bigg(1+\frac{1}{\varepsilon}\Bigg)-\dfrac{c(n+1)}{2\lambda_1(M)}\Bigg(1+\frac{1}{\varepsilon}\Bigg)+\frac{2}{\varepsilon}\Bigg]\\
&\hspace{1.6cm}\int_M\left|\nabla\varphi\right|^2\left(\left|A\right|+\delta\right)^{p-2}\left|A\right|^2\\
&\hspace{1.5cm}+\Bigg\{C_s\|A\|_n^2\Bigg(4-\frac{2}{n}\Bigg)\left(1+\varepsilon\right)\Bigg[\left(1+\eta\right)+\left(1+\frac{1}{\eta}\right)\frac{(p-2)^2}{4}\Bigg]\\
&\hspace{1.5cm}-\dfrac{c(n+1)}{2\lambda_1(M)}(1+\varepsilon)\Bigg[\left(1+\eta\right)+\left(1+\frac{1}{\eta}\right)\frac{(p-2)^2}{4}\Bigg]+2\varepsilon-2(p-2)\Bigg\}\\
&\hspace{1.6cm}\int_{M}\varphi^2\left(\left|A\right|+\delta\right)^{p-2}\left|\nabla\left|A\right|\right|^2.
\end{align*}
Now we choose $\eta=\frac{2-p}{p}$ and the previous inequality implies
\begin{align*}
&\Bigg\{\Bigg(2p-2\varepsilon+\frac{4}{n}-2\Bigg)+\Bigg[C_s\|A\|_n^2\Bigg(\frac{2}{n}-4\Bigg)\left(1+\varepsilon\right)+\dfrac{c(n+1)}{2\lambda_1(M)}(1+\varepsilon)\Bigg]\frac{(4-p)^2}{4}\Bigg\}\\
&\hspace{0.5cm}\int_{M}\varphi^2\left(\left|A\right|+\delta\right)^{p-2}\left|\nabla\left|A\right|\right|^2\\
&\leq\Bigg[C_s\|A\|_n^2\Bigg(4-\frac{2}{n}\Bigg)\left(1+\frac{1}{\varepsilon}\right)-\dfrac{c(n+1)}{2\lambda_1(M)}\left(1+\frac{1}{\varepsilon}\right)+\frac{2}{\varepsilon}\Bigg]\int_M\left|\nabla\varphi\right|^2\left|A\right|^p.
\end{align*}
It is easy to see that 
\begin{align*}
&\Bigg(2p-2\varepsilon+\frac{4}{n}-2\Bigg)+\Bigg[C_s\|A\|_n^2\Bigg(\frac{2}{n}-4\Bigg)\left(1+\varepsilon\right)+\dfrac{c(n+1)}{2\lambda_1(M)}(1+\varepsilon)\Bigg]\frac{(4-p)^2}{4}\\
&\overset{\varepsilon\to 0}{\longrightarrow}\Bigg(2p+\frac{4}{n}-2\Bigg)+\Bigg[C_s\|A\|_n^2\Bigg(\frac{2}{n}-4\Bigg)+\dfrac{c(n+1)}{2\lambda_1(M)}\Bigg]\frac{(4-p)^2}{4}.
\end{align*}
Now, we assume that
\begin{align}\label{lambdacondition}
\lambda_1(M)
>\dfrac{-c(n+1)(4-p)^2}{16\Big(p+\frac{2}{n}-1\Big)}
\end{align}
and 
\begin{equation}\label{sff}
\|A\|_n^2<\frac{8\Big(p+\frac{2}{n}-1\Big)+\frac{c(n+1)(4-p)^2}{2\lambda_1(M)}}{C_s\left(4-\frac{2}{n}\right)(4-p)^2}.
\end{equation}
Using the above assumptions, there exists $\epsilon>0$ and a constant $C=C(\epsilon, n,\lambda_1(M), \|A\|_n)$ such that
$$\int_M\varphi^2\left(\left|A\right|+\delta\right)^{p-2}|\nabla|A||^{2}\leq C\int_M  \left|\nabla\varphi\right|^2\left|A\right|^{p}.$$
Now, for any $R>0$ and a fixed point $o\in M$, we choose $\varphi\in\mathcal{C}^\infty_0(M)$ satisfying $0\leq \varphi\leq 1, \varphi\equiv1$ in the geodesic ball $B_o(R)$, $\varphi\equiv 0$ outside $B_o(2R)$, and $|\nabla\varphi|\leq 1/R$. Plugging $\varphi$ in the above inequality, by letting $R\to\infty$, we conclude that
$$\int_M|\nabla|A||^{2}=0.$$
This implies $|\nabla|A||=0$, consequently, $|A|$ is constant. Since a Sobolev inequality holds true on $M$, the volume of $M$ must be infinite. Hence, $|A|=0$ due to $\|A\|_n<\infty$. To prove the Theorem \ref{rigid}, we note that

\textbf{Case 1: }If $c=-1$, then inequalities \eqref{lambdacondition} and \eqref{sff} mean
$$
\lambda_1(M)
>\dfrac{(n+1)(4-p)^2}{16\Big(p+\frac{2}{n}-1\Big)}
$$
and 
$$
\|A\|_n^2<\frac{8\Big(p+\frac{2}{n}-1\Big)-\frac{(n+1)(4-p)^2}{2\lambda_1(M)}}{C_s\left(4-\frac{2}{n}\right)(4-p)^2}.
$$

\textbf{Case 2: }If $c=0$, then inequality \eqref{lambdacondition} can be removed.  Moreover, \eqref{sff} means
$$
\|A\|_n^2<\frac{8\Big(p+\frac{2}{n}-1\Big)}{C_s\left(4-\frac{2}{n}\right)(4-p)^2}.
$$
 The proof is complete. 
\end{proof}

\noindent
Note that sectional curvatures of $\mathbb{CH}^n(-1)$ are in $\left[-1, -\frac{1}{4}\right]$.  Therefore, when $M$ is minimal \textcolor{blue}{Theorem \ref{lambda}} gives us a lower bound of $\lambda_1(M)$. 
$$\lambda_1(M)\geq\frac{(n-1)^2}{16}.$$
Using this estimation, following the proof of Theorem \ref{rigid}, we see that the inequality \eqref{sff} becomes
$$p+\frac{2}{n}-1-\frac{(n+1)(4-p)^2}{(n-1)^2}>0 \text{ and }\|A\|_n< \sqrt{\frac{4\Big(p+\frac{2}{n}-1\Big)-\frac{4(n+1)(4-p)^2}{(n-1)^2}}{C_s\left(4-\frac{2}{n}\right)(4-p)^2}}.$$
As a consequence, we obtain the following result.
\begin{proposition}
Let $M$ be an $n$-dimensional totally real minimal submanifold immersed in $\mathbb{CH}^n(-1)$. For any $1<p<2$ such that 
$$p+\frac{2}{n}-1-\frac{(n+1)(4-p)^2}{(n-1)^2}>0.$$
If $$\|A\|_n< \sqrt{\frac{4\Big(p+\frac{2}{n}-1\Big)-\frac{4(n+1)(4-p)^2}{(n-1)^2}}{C_s\left(4-\frac{2}{n}\right)(4-p)^2}}$$
and $\|A\|_{p}<\infty$, then $M$ is totally geodesic. 
\end{proposition}
\noindent
\section{Totally real minimal submanifolds and Connectedness at infinity}
This section is used to give a proof of Theorem \ref{ends}.
\begin{proof}[Proof of Theorem \ref{ends}]
Let $\omega$ be an $L^{p}$ harmonic $1$-form on $M$. It is well-known that (for instance, see \cite{CDS24}) 
$${\rm Ric}\geq \frac{1}{4}(n-1)c-|A|^2.$$
Therefore, an application of Bochner's formula yields
$$\begin{aligned} 
\Delta\left|\omega\right|^2
&=2(\left|\nabla\omega\right|^2+{\rm Ric}(\omega,\omega))\\
&\geq 2|\nabla\omega|^2+2\left(\frac{1}{4}(n-1)c-|A|^2\right)|\omega|^2.
\end{aligned} $$
Using the refined Kato inequality and the fact that $  
\Delta\left|\omega\right|^2=2(\left|\omega\right|\nabla\left|\omega\right|+\left|\nabla\left|\omega\right|\right|^2)
 $, we have 
\begin{equation}\label{eq:13b}
\left|\omega\right|\Delta\left|\omega\right|\geq \dfrac{1}{n-1}\left|\nabla\left|\omega\right|\right|^2-\left|A\right|^2\left|\omega\right|^2+\dfrac{1}{4}\Big(n-1\big)c\left|\omega\right|^2.
\end{equation}
Let $\varphi$ be a Lipschitz function with compact support in a geodesic ball $B_o(3R)$ of radius $R$ centered at $o\in M$. For a positive number $\delta>0$, multiplying both sides of the above inequality by $\varphi^2\left(\left|\omega\right|+\delta\right)^{p-2}$ and integrating over $B(R)$, we obtain
\begin{align*}
\int_{B(R)}\varphi^2\left(\left|\omega\right|+\delta\right)^{p-2}\left|\omega\right|\Delta\left|\omega\right|-\dfrac{c}{4}(n-1)& \int_{B(R)}\varphi^2\left(\left|\omega\right|+\delta\right)^{p-2}\left|\omega\right|^{2}\\
+\int_{B(R)}\varphi^2\left(\left|\omega\right|+\delta\right)^{p-2}\left|A\right|^{2}\left|\omega\right|^{2}
&\geq \dfrac{1}{n-1}\int_{B(R)}\varphi^2\left(\left|\omega\right|+\delta\right)^{p-2}\left|\nabla\left|\omega\right|\right|^2.
\end{align*}
An application of the divergence theorem yields
\begin{align*}
\int_{B(R)}&\varphi^2\left(\left|\omega\right|+\delta\right)^{p-2}\left|\omega\right|\Delta\left|\omega\right|\\
&=-\int_{B(R)}\left\langle\nabla\left[\varphi^2\left(\left|\omega\right|+\delta\right)^{p-2}\left|\omega\right|\right], \nabla \left|\omega\right|\right\rangle\\
&=-\left(p-2\right)\int_{B(R)}\varphi^2\left|\omega\right|\left(\left|\omega\right|+\delta\right)^{p-3}\left|\nabla\left|\omega\right|\right|^2-\int_{B(R)}\varphi^2\left(\left|\omega\right|+\delta\right)^{p-2}\left|\nabla\left|\omega\right|\right|^2\\
&\hspace{0.45cm}-2\int_{B(R)}\varphi\left|\omega\right|\left(\left|\omega\right|+\delta\right)^{p-2}\left\langle \nabla \varphi, \nabla\left|\omega\right|\right\rangle.
\end{align*}
This identity together with the previous inequality implies,  
\begin{equation}\label{eq:15}
\begin{split}
\dfrac{n}{n-1}&\int_{B(R)}\varphi^2\left(\left|\omega\right|+\delta\right)^{p-2}\left|\nabla\left|\omega\right|\right|^2\\
&\leq -\dfrac{c}{4}(n-1)\int_{B(R)}\varphi^2\left(\left|\omega\right|+\delta\right)^{p-2}\left|\omega\right|^{2}\\
&\hspace{0.44cm}+\int_{B(R)}\varphi^2\left(\left|\omega\right|+\delta\right)^{p-2}\left|A\right|^{2}\left|\omega\right|^{2}\\
&\hspace{0.43cm}-\left(p-2\right)\int_{B(R)}\varphi^2\left|\omega\right|\left(\left|\omega\right|+\delta\right)^{p-3}\left|\nabla\left|\omega\right|\right|^2\\
&\hspace{0.43cm}-2\int_{B(R)}\varphi^2\left|\omega\right|\left(\left|\omega\right|+\delta\right)^{p-2}\left\langle \nabla \varphi, \nabla\left|\omega\right|\right\rangle.
\end{split}
\end{equation}
Using the same strategy as in the proof of Theorem \ref{rigid}, by applying H$\ddot{o}$lder's inequality, the weighted Sobolev's inequality and Cauchy-Schwarz's inequality consecutively, we obtain
\begin{equation}\label{eq:17}
\begin{split}
\int_{B(R)}&\varphi^2\left(\left|\omega\right|+\delta\right)^{p-2}\left|A\right|^{2}\left|\omega\right|^{2}\\
&\leq \Bigg(\int_{B(R)}\left|A\right|^{2\cdot\frac{n}{2}}\Bigg)^{\frac{2}{n}}\Bigg\{\int_{B(R)}\left[\left|\omega\right| \varphi\left(\left|\omega\right|+\delta\right)^{\frac{p-2}{2}}\right]^{2\cdot\frac{n}{n-2}}\Bigg\}^{\frac{n-2}{n}}\\
&\leq C_s\|A\|_n^2\int_{B(R)}\left|\nabla\left[|\omega|\varphi\left(\left|\omega\right|+\delta\right)^{\frac{p-2}{2}}\right]\right|^2\\
&\leq C_s\|A\|_n^2\left(1+\frac{1}{\varepsilon}\right)\int_{B(R)}\left|\nabla \varphi\right|^2\left(\left|\omega\right|+\delta\right)^{p-2}\left|\omega\right|^2\\
&\hspace{0.5cm}+C_s\|A\|_n^2\left(1+\varepsilon\right)\Bigg[\left(1+\eta\right)+\left(1+\frac{1}{\eta}\right)\frac{(p-2)^2}{4}\Bigg]\\
&\hspace{0.6cm}\int_{B(R)}\varphi^2\left(\left|\omega\right|+\delta\right)^{p-2}\left|\nabla\left|\omega\right|\right|^2.
\end{split}
\end{equation}
Here we used Lemma \ref{abc} in the last inequality. Note that if $c=0$ then the first term on the right hand side of \eqref{eq:15} disappears. If $c=-1$ then the sectional curvatures of $M$ are in $\left[-1, -\frac{1}{4}\right]$, we can estimate this term as follows. Since $M$ is minimal, Theorem \ref{lambda} infers
$$\lambda_1(M)\geq\frac{(n-1)^2}{16}.$$
Using this inequality, we estimate the first term on the right hand side of \eqref{eq:15} as follows.
\begin{align}
\frac{(n-1)^2}{16}&\int_{B(R)}\left|\omega\right|^{2}\varphi^2\left(\left|\omega\right|+\delta\right)^{p-2}\notag\\
&\leq\int_{B(R)}\left|\nabla\left[|\omega|\varphi\left(\left|\omega\right|+\delta\right)^{\frac{p-2}{2}}\right]\right|^2\notag\\ 
&\leq \left(1+\frac{1}{\varepsilon}\right)\int_{B(R)}\left|\nabla \varphi\right|^2\left(\left|\omega\right|+\delta\right)^{p-2}\left|\omega\right|^2\\
&\hspace{0.5cm}+\left(1+\varepsilon\right)\Bigg[\left(1+\eta\right)+\left(1+\frac{1}{\eta}\right)\frac{(p-2)^2}{4}\Bigg]\notag\\
&\hspace{0.6cm}\int_{B(R)}\varphi^2\left(\left|\omega\right|+\delta\right)^{p-2}\left|\nabla\left|\omega\right|\right|^2.\notag
\end{align}
Again, we observe that 
\begin{equation}\label{eq:12b}
\int_{B(R)}f^2\left(\left|\omega\right|+\delta\right)^{p-2}\left|\nabla\left|\omega\right|\right|^2 \geq\int_{B(R)}f^2\left(\left|\omega\right|+\delta\right)^{p-3}\left|\omega\right|\left|\nabla\left|\omega\right|\right|^2. 
\end{equation}
and 
\begin{equation}\label{eq:13c}
\begin{split}
-2\int_{B(R)}&\varphi^2\left(\left|\omega\right|+\delta\right)^{p-2}\left|\omega\right|\langle\nabla \varphi,\nabla\left|\omega\right|\rangle\\
&\leq\frac{1}{\epsilon}\int_{B(R)}\left(\left|\omega\right|+\delta\right)^{p-2}\left|\omega\right|^{2}\left|\nabla \varphi\right|^2
+\epsilon\int_{B(R)}\varphi^2\left(\left|\omega\right|+\delta\right)^{p-2}\left|\nabla\left|\omega\right|\right|^2  
\end{split}
\end{equation}
for any $\epsilon>0$. Now, let $\eta=\frac{2-p}{p}$ then combining \eqref{eq:15}-\eqref{eq:13c}, we obtain
\begin{align}\label{28e1} 
&\Bigg[\dfrac{n}{n-1}-\varepsilon+p-2-C_s\|A\|_n^2(1+\varepsilon)\frac{(4-p)^2}{4}+\frac{4c}{n-1}(1+\varepsilon)\frac{(4-p)^2}{4}\Bigg]\\
&\int_{B(R)}\varphi^2\left(\left|\omega\right|+\delta\right)^{p-2}\left|\nabla\left|\omega\right|\right|^2\leq \Bigg[C_s\|A\|_n^2\Bigg(1+\frac{1}{\varepsilon}\Bigg)-\frac{4c\Big(1+\frac{1}{\varepsilon}\Big)}{n-1}+\dfrac{1}{\epsilon}\Bigg]\int_{B(R)}\left|\nabla \varphi\right|^2\left|\omega\right|^{p}.\notag 
\end{align}
\textbf{Case 1: }$c=-1$.

\noindent
Observe that the condition
\begin{equation}\label{kq1}
\dfrac{n-(4-p)^2}{n-1}+p-2-C_s\|A\|_n^2\frac{(4-p)^2}{4}>0
\end{equation}
is equivalent to
$$\frac{4n-4(4-p)^2+4(p-2)(n-1)}{(n-1)(4-p)^2}>C_s\|A\|_n^2.$$
Then, our assumptions on $\|A\|$ means the condition \eqref{kq1} is satisfied. Therefore, we can choose a sufficiently small $\epsilon>0$ such that
$$ 
\dfrac{n}{n-1}-\varepsilon+p-2-C_s\|A\|_n^2(1+\varepsilon)\frac{(4-p)^2}{4}-\frac{4}{n-1}(1+\varepsilon)\frac{(4-p)^2}{4}>0.
$$
As a consequence, the inequality \eqref{28e1} implies that there exists a positive constant $C=C(n,p,\epsilon, \|A\|_n)$ such that
$$\int_M(|\omega|+\delta)^{p-2}|\nabla|\omega||^2\varphi^2\leq C\int_M|\omega|^{2}|\nabla \varphi|^p.$$
 Now, for a fixed point $o\in M$ and $R>0$, we choose $\varphi\in\mathcal{C}^\infty_0(M)$ satsifying $0\leq \varphi\leq1$, $\varphi\equiv1$ in the geodesic ball $B_o(R)$, $\varphi\equiv0$ outside $B_o(2R)$, and $|\nabla \varphi|\leq\frac{4}{R}$. The above inequality infers
 $$\int_{B_o(R)}(|\omega|+\delta)^{p-2}|\nabla|\omega||^2\leq \frac{C}{R^2}\int_{B_o(2R)}|\omega|^{p}.$$
Note that $\int_{M}\left| \omega\right|^{p}<\infty$, letting $R$ tend to infinity, we infer
 $\left|\nabla\left|\omega\right|\right|\equiv 0$. Consequently, $\left| \omega \right| \equiv$ constant. Note that $M$ is of infinite volume since the Sobolev inequality holds true there. This together with $\int_{M}|\omega|^{p}<\infty$ implies $|\omega|\equiv 0$ or equivalently $\omega=0$.

\noindent
\textbf{Case 2: } $c=0$. 

\noindent
In this case, the inequality \eqref{28e1} becomes 
\begin{align}
&\Bigg[\dfrac{n}{n-1}-\varepsilon+p-2-C_s\|A\|_n^2(1+\varepsilon)\frac{(4-p)^2}{4}\Bigg]\int_{B(R)}\varphi^2\left(\left|\omega\right|+\delta\right)^{p-2}\left|\nabla\left|\omega\right|\right|^2\notag\\
&\leq \Bigg[C_s\|A\|_n^2\Bigg(1+\frac{1}{\varepsilon}\Bigg)+\dfrac{1}{\epsilon}\Bigg]\int_{B(R)}\left|\nabla \varphi\right|^2\left|\omega\right|^{p}.\notag 
\end{align}
As in the previous case, we want to have that  
\begin{equation}\label{kq2}
\dfrac{n}{n-1}+p-2-C_s\|A\|_n^2\frac{(4-p)^2}{4}>0,
\end{equation}
or equivalently, 
$$\frac{4n+4(p-2)(n-1)}{(n-1)(4-p)^2}>C_s\|A\|_n^2.$$
Now, we can repeat the argument of the previous part to complete the proof. We omit the details. 
\end{proof}


\section*{Acknowledgement}
This work was completed during the stay of the first and the fourth authors at Vietnam Institute for Advanced Study in Mathematics (VIASM), in Summer 2025. They would like to thank the staff there for hospitality and support.


\address{{\it Nguyen Thac Dung}\\
Faculty of Mathematics, Informatics, and Mechanics\\
Vietnam National Univeristy\\
University of Science at Hanoi, Vietnam}
{dungmath@gmail.com; or dungmath@vnu.edu.vn}

\address{{\it Le Gia Linh}\\
Department of Mathematics, Informatics, and Mechanics\\
Vietnam National Univeristy\\
University of Science at Hanoi, Vietnam}
{linhgiale1999@gmail.com}

\address{{\it Phung Bich Ngan}\\
Department of Mathematics,\\
 Hanoi Pedagogical University 2 \\
No. 32 Nguyen Van Linh, Xuan Hoa, Phuc Yen, Vinh Phuc}
{nganbg50@gmail.com}

\address{{ \it Abhitosh Upadhyay}\\
School of Mathematics and Computer Science,\\ 
Indian Institute of Technology\\ 
 Goa, Ponda, 403401, India}
{abhitosh@iitgoa.ac.in}
\end{document}